\documentclass[11pt]{article}
\usepackage{mathrsfs}
\usepackage{amsfonts}
\usepackage{amsthm}
\usepackage{mathrsfs}
\usepackage{amsmath, amssymb}
\usepackage{shortvrb,psfrag}
\usepackage{enumerate}
\usepackage{color}
\usepackage{graphicx}
\usepackage{dsfont}
\usepackage{latexsym}
\usepackage{longtable}

\newtheorem{athm}{\bf \t}[section]
\newenvironment{thm} [1] {\def\t{#1}\begin{athm} \bf \rm} {\end{athm}}
\newcommand{\bthm}{\begin{thm}}\newcommand{\ethm}{\end{thm}}

\newtheorem{lemma}{Lemma}[section]
\newtheorem{remark}{Remark}[section]

\newtheorem{proposition}{Proposition}[section]

\newcommand{\beq}{\begin{equation}}
\newcommand{\eeq}{\end{equation}}
\newcommand{\ben}{\begin{eqnarray}}
\newcommand{\een}{\end{eqnarray}}
\newcommand{\beno}{\begin{eqnarray*}}
\newcommand{\eeno}{\end{eqnarray*}}
\newcommand{\bali}{\begin{aligned}}
\newcommand{\eali}{\end{aligned}}

\numberwithin{equation}{section}

\newcommand{\al}{\alpha}

\newcommand{\ve}{\varepsilon}
\newcommand{\ga}{\gamma}

\newcommand{\f}{\frac}
\newcommand{\na}{\nabla}

\newcommand{\ud}{\mathrm{d}}
\newcommand{\ue}{\mathrm{e}}
\newcommand{\uu}{\mathbf{u}}
\newcommand{\vv}{\mathbf{v}}

\newcommand{\xx}{\mathbf{x}}
\newcommand{\yy}{\mathbf{y}}
\newcommand{\nn}{\mathbf{n}}

\newcommand{\hh}{\mathbf{h}}
\newcommand{\mm}{\mathbf{m}}

\newcommand{\MM}{\mathbf{M}}
\newcommand{\NN}{\mathbf{N}}

\newcommand{\DD}{\mathbf{D}}
\newcommand{\FF}{\mathbf{F}}
\newcommand{\II}{\mathbf{I}}

\newcommand{\CR}{\mathcal{R}}
\newcommand{\CF}{\mathcal{F}}

\newcommand{\CQ}{\mathcal{Q}}
\newcommand{\CL}{\mathcal{L}}
\newcommand{\CM}{\mathcal{M}}
\newcommand{\CK}{\mathcal{K}}
\newcommand{\CJ}{\mathcal{J}}
\newcommand{\FS}{\mathfrak{S}}

\newcommand{\BS}{{\mathbb{S}^2}}
\newcommand{\BR}{{\mathbb{R}^3}}
\newcommand{\BOm}{\mathbf{\Omega}}

\newcommand{\Ae}{{A_{\varepsilon}}}
\newcommand{\Ue}{{U_{\varepsilon}}}
\newcommand{\mue}{{\mu_{\varepsilon}}}
\newcommand{\Be}{{B_{\varepsilon}}}
\newcommand{\sqe}{{\sqrt{\varepsilon}}}
\newcommand{\pa}{\partial}

\begin{document}

\title{From microscopic theory to macroscopic theory: dynamics of the rod-like liquid crystal molecules}

\author{Wei Wang$^\dag$,\, Pingwen Zhang $^\ddag$\,and \,Zhifei Zhang$^\ddag$\\[2mm]
{\small $^\dag$ Beijing International Center for Mathematical Research, Peking University, P. R. China}\\
{\small $^\ddag$ LMAM \& School of  Mathematical Sciences, Peking University, P. R. China}\\
{\small E-mail: wangw07@pku.edu.cn, pzhang@pku.edu.cn,
zfzhang@math.pku.edu.cn}}

\maketitle

\begin{abstract}
Starting from Doi-Onsager equation for the liquid crystal, we first derive the Q-tensor equation by the Bingham closure.
Then we derive the Ericksen-Leslie equation from the Q-tensor equation by taking the small Deborah number limit.
\end{abstract}

\section{Introduction}

Liquid crystals are a state of matter that have properties between
those of a conventional liquid and those of a solid crystal. One of
the most common liquid crystal phases is the nematic. The nematic liquid crystals are composed of rod-like molecules with the long
axes of neighboring molecules approximately aligned to one another. There are three different kinds of theories to model the nematic liquid crystals.

\subsection{Doi-Onsager theory}

The state of alignment of the nematic liquid crystal molecules(LCP) is described by the orientational distribution function.
A classic model that predicts isotropic-nematic phase transition is
the hard-rod model proposed by Onsager \cite{On}, in which the rod-rod interaction is modeled by the
excluded volume effect. Maier and Saupe \cite{MS} following Onsager
proposed a slightly modified interaction potential, now known as the
Maier-Saupe potential. Doi and Edwards \cite{Doi} extended Onsager's
theory in order to describe the behavior of liquid crystal polymer flows.

We use $\xx\in \BR$ to denote the material point and
$f(\xx,\mm,t)$ to represent the number density for the number of
molecules whose orientation is parallel to $\mm$ at point $\xx$ and
time $t$. For the spatially homogeneous liquid crystal flow,
the Doi-Onsager equation \cite{Doi} takes
\begin{eqnarray}\label{eq:Doi-Onsager}
\frac{\partial{f}}{\partial{t}}=\f 1 {De}\CR\cdot(\CR{f}+f\CR
U)-\CR\cdot\big(\mm\times\kappa\cdot\mm{f}\big),
\end{eqnarray}
where $De$ is the Deborah number, $\CR=\mm\times\nabla_\mm$ is the rotational gradient
operator, $\kappa$ is a constant velocity gradient, and $U$ is the
mean-field interaction potential. This model
has a free energy
\begin{equation}\label{eq:free energy-homo}
A[f]=\int_{\BS}\big(f(\mm,t)\ln{f(\mm,t)}+\frac12f(\mm,t)U(\mm,t)\big)\ud\mm
\end{equation}
as its Lyapunov functional.

The homogeneous Doi-Onsager equation has been very successful in describing the
properties of liquid crystal polymers in a solvent. This model takes
into account the effects of hydrodynamic flow, Brownian motion, and
intermolecular forces on the molecular orientation distribution.
However, it does not include effects such as distortional
elasticity. Therefore, it is valid only in the limit of spatially
homogeneous flows.

Inhomogeneous flows were first studied by Marrucci and Greco
\cite{MG}, and subsequently by many people \cite{FLS, Wang}. Instead
of using the distribution as the sole order parameter, they used a
combination of the tensorial order parameter and the distribution
function, and used the spatial gradients of the tensorial order
parameter to describe the spatial variations. This is a departure
from the original motivation that led to the kinetic theory.
Wang, E, Liu and Zhang \cite{WELZ} set up a formalism in which the
interaction between molecules is treated more directly using the
position-orientation distribution function via interaction
potentials. They extended the free energy (\ref{eq:free energy-homo})
to include the effects of nonlocal intermolecular interactions
through an interaction potential as follows:
\begin{align}
&\Ae[f]=\int_{\BR}\int_{\BS}f(\xx,\mm,t)(\ln{f(\xx,\mm,t)}-1)
+\frac{1}{2}\Ue(\xx,\mm,t)f(\xx,\mm,t)\ud\mm\ud\xx,\\
&\Ue(\xx,\mm,t)=\int_{\BR}\int_{\BS}\Be(\xx,\xx',\mm,\mm')f(\xx',\mm',t)\ud\mm'\ud\xx'.
\end{align}
Here $\Be(\xx,\xx';\mm,\mm')$ is the interaction kernel.
There are two typical choices:
\begin{enumerate}
\item Long range Maier-Saupe interaction potential:
$$\Be(\xx,\xx',\mm,\mm')=\frac{1}{\ve^{3/2}}g(\frac{\xx-\xx'}{\sqrt{\ve}})\al|\mm\times\mm'|^2,$$
where $g(\xx)\in C^{\infty}(\BR)$ is a radial function with $\int_\BR g(\xx)\ud{\xx}=1$,
and the small parameter $\sqe$ represents the typical interaction distance.

\item Hard-core interaction potential:
\begin{align*}
\Be(\xx,\xx',\mm,\mm')=
\left\{
\begin{array}{ll}
0, \quad{\small \text{molecule} (\mathbf{x},\mathbf{m}) \text{ is disjoint with
molecule } (\mathbf{x}',\mathbf{m}')},\\
1, \quad\text{\small joint with each other}.
\end{array}
\right.
\end{align*}
\end{enumerate}
In this paper, we first focus on the Maier-Saupe case. In the last section,
we will discuss the hard-core potential case.

The chemical potential $\mue$ is defined as
$$\mue=\frac{\delta{\Ae[f]}}{\delta{f}}=\ln{f(\xx,\mm,t)+\Ue(\xx,\mm,t)}.$$
We also introduce non-interacting kernel
$$B_0(\xx,\xx',\mm,\mm')=\alpha|\mm\times\mm'|^2\delta(\xx-\xx').$$
We denote by $U_0,~\mu_0,~A_0[f]$ the corresponding non-interacting potential, chemical potential and free energy respectively.
Throughout this paper, we use the notation
$$\langle(\cdot)\rangle_f=\int_{\BS}(\cdot) f(\xx,\mm,t)\ud{\mm}.$$

The non-dimensional Doi-Onsager equation takes the form:
\begin{align}
\frac{\pa{f}}{\pa{t}}+\vv\cdot\nabla{f}=&\frac{\ve}{De}\nabla\cdot\big\{
\big(\gamma_{\|}\mm\mm+\gamma_{\bot}(\II-\mm\mm)\big)\cdot(\nabla{f}+f\nabla{\Ue})\big\}\nonumber\\
&+\frac{1}{De}\CR\cdot(\CR{f}+f\CR{\Ue})
-\CR\cdot(\mm\times\kappa\cdot\mm{f}),\label{Doi-v}\\\nonumber
\frac{\pa{\vv}}{\pa{t}}+\vv\cdot\nabla\vv=&-\nabla{p}+\frac{\gamma}{Re}\Delta\vv
+\frac{1-\gamma}{2Re}\nabla\cdot(\DD:\langle\mm\mm\mm\mm\rangle_f)+\frac{1-\gamma}{De
Re}(\nabla\cdot\tau^e+\FF^e),\\
\nabla\cdot\vv=&0,\nonumber
\end{align}
where $\kappa=(\nabla\vv)^T,~\DD=\frac{1}{2}(\kappa+\kappa^T)$, while $\tau^e$ and $\FF^e$ represent the elastic stress and body force
respectively defined as
\begin{eqnarray}
\tau^e=-\langle\mm\mm\times\CR\mue\rangle_f=(3\langle\mm\mm-\frac{1}{3}\II\rangle_f
-\langle\mm\mm\times\CR{\Ue}\rangle_f),\quad
\FF^e=-\langle\nabla\mue\rangle_f.
\end{eqnarray}
The constants $\ga_{\|}$ and $\ga_{\bot}$  denote the translational diffusion coefficients parallel to and
normal to the orientation of the LCP molecule respectively. In the case when $\ga_{\|}=\ga_{\bot}=0$,
we may assume that $\int_{\BS}f(\xx,\mm)\ud\mm=1,$ which means that the density of the molecular is constant.

The Doi-Onsager equation (\ref{Doi-v}) has the following the energy dissipation relation:
\begin{eqnarray}
&&-\frac{\ud}{\ud{t}}\Big[\int_{\BR}\frac{1}{2}|\vv|^2\ud\xx+\frac{1-\gamma}{DeRe}A_\ve[f]\Big]\nonumber
\\&&=\int_{\BR}\Big(\frac{\gamma}{Re}\DD:\DD+\frac{1-\gamma}{2Re}\langle(\mm\mm:\DD)^2\rangle_f
+\frac{1-\gamma}{De^2Re}\langle\CR\mu\cdot\CR\mu\rangle_f\nonumber\\
&&\quad+\frac{\ve}{De^2Re}\langle\nabla\mu\cdot(\gamma_{\|}\mm\mm+\gamma_\bot(\II-\mm\mm))\cdot\nabla\mu\rangle_f
\ud\Big)\ud\xx.
\end{eqnarray}

\subsection{Landau-de Gennes theory(Q-tensor theory)}

One of continuum theory for the nematic liquid crystals is the Landau-de Gennes theory \cite{DG}.
In this theory, the sate of the nematic liquid crystals is described by the macroscopic Q-tensor order parameter,
which is a symmetric, traceless $3\times 3$ matrix. Physically, it can be interpreted as the second-order moment of
the orientational distribution function $f$, that is,
\beno
Q=\int_\BS(\mm\mm-\frac{1}{3}\II)f\ud\mm.
\eeno
When $Q=0$, the nematic liquid crystal is said to be isotropic. When $Q$ has two equal non-zero eigenvalues, it is said to be unixial. In such case,
$Q$ can be written as
\beno
Q=s\big(\nn\nn-\f13\II\big),\quad \nn\in \BS.
\eeno
When $Q$ has three distinct eigenvalues, it is said to be biaxial. In such case, $Q$ can be written as
\beno
Q=s\big(\nn\nn-\f13\II\big)+\lambda(\nn'\nn'-\frac13\II),\quad \nn,\nn'\in \BS,\quad \nn\cdot\nn'=0.
\eeno

The Landau-de Gennes energy functional $E_{LG}$ is given by
\ben
E_{LG}=\int_{\BR}w(Q,\na Q)+f_{bulk}(Q)\ud \xx,
\een
where $w$ is the elastic energy, which in general takes
\ben
w(Q,\nabla Q)=L_1|\nabla Q|^2+L_2Q_{ik,j}Q_{ij,k}+L_3Q_{ij,j}Q_{ik,k}+L_4Q_{kl}Q_{ij,k}Q_{ij,l},
\een
where $Q=(Q_{ij})$, $Q_{ij,k}=\frac{\partial Q_{ij}}{\partial x_k}$ and $L_1,\cdots,L_4$ are elastic constants;
$f_{bulk}$ is the bulk energy which in the simplest form takes
\ben
f_{bulk}(Q)=a\mathrm{tr}(Q^2)+\frac{2b}{3}\mathrm{tr}(Q^3)+\frac{c}{2}\big(\mathrm{tr}(Q^2)\big)^2,
\een
where $a, b, c$ are temperature dependent constants. We refer to \cite{MN} for more details.

There are two popular dynamical $Q$-tensor models for liquid crystals, which are derived
by Beris-Edwards \cite{BE} and Qian-Sheng \cite{QS} separately.
When the total energy in $Q$-tensor form
is given by $E(Q,\nabla Q)$, we define
$$\mu_{Q}=\frac{\delta{E(Q,\nabla Q)}}{\delta Q}.$$
The dynamical $Q$-tensor theory could written in
the following form in general:
\begin{align}
\frac{\pa{{Q}}}{\pa{t}}+\vv\cdot\nabla{Q}&=D^{rot}(\mu_Q)+F(Q,\DD)
+\BOm\cdot\mu_Q-\mu_Q\cdot\BOm,\nonumber\\
\frac{\pa{\vv}}{\pa{t}}+\vv\cdot\nabla\vv&=-\nabla{p}+\nabla\cdot\big(\sigma^{dis}+\sigma^{s}+\sigma^a+\sigma^{d}\big),
\label{eq:Q-general-intro}\\\nonumber
\nabla\cdot\vv=0,
\end{align}
where $D^{rot}(\mu_Q)$ is the rotational diffusion term,
$F(Q,\DD)$ is the velocity-induced term, and $\sigma^d$ is the distortion stress,
$\sigma^a$ is the anti-symmetric part of orientational-induced stress,
$\sigma^{s}=\gamma F(Q, \mu_Q)$ is the symmetric stress induced by the orientational,
which conjugates to $F(Q, \DD)$ ($\gamma$ is a constant), $\sigma^{dis}$
is the additional dissipation stress.

In  Beris-Edwards's model and Qian-Shen's model, module some constants, $\sigma^a$ and $\sigma^d$ are the same, i.e.,
\begin{align}
\sigma_{ij}^d=\frac{\partial E}{\partial (Q_{kl,j})}Q_{kl,i},\quad
\sigma^a=Q\cdot\mu_Q-\mu_Q\cdot Q.
\end{align}
In Beris- Edwards's model, the other terms are given by
\begin{align*}
&D^{rot}_{BE}=-\Gamma \mu_Q,\quad \sigma_{BE}^{dis}=0,\quad \sigma_{BE}^s=F_{BE}(Q,\mu_Q),\\
&F_{BE}(Q,A)=\xi\Big((Q+\frac13Id)\cdot A+A\cdot(Q+\frac13Id)-2(Q+\frac13Id)(A:Q)\Big).
\end{align*}
In Qian-Sheng's model, they are given by
\begin{align*}
&D^{rot}_{QS}=-\Gamma \mu_Q,\quad
\sigma_{QS}^s=-\frac12\frac{\mu_2^2}{\mu_1}\mu_{Q},\quad
F_{QS}(Q,\DD)=-\frac12\frac{\mu_2}{\mu_1}\DD,\\
&\sigma^{dis}_{QS}=\beta_1 Q(Q:A)+\beta_2 \DD+ \beta_3(Q\cdot\DD+\DD\cdot Q).
\end{align*}

There are also other dynamic models by using $Q$ tensor to describe the flow of the nematic liquid crystals, which are obtained
by various closure approximations or the variational principle. We refer to \cite{FLS, Feng, SMV} and references therein.

\subsection{Ericksen-Lesilie theory}

The Ericksen-Leslie theory \cite{Eri61, Eri,Les} is an elastic continuum
theory, which is a very powerful tool for modeling liquid crystal devices.
This theory treats the liquid crystal material as a continuum and
completely ignores molecular details. Moreover, this theory considers
perturbations to a presumed oriented sample.

In this theory, the configuration of the liquid crystals is described by a director
field $\nn(\xx, t)$. The Ericksen-Leslie equation takes the form
\begin{eqnarray}\label{eq:EL}
\left\{
\begin{split}
&\vv_t+\vv\cdot\nabla\vv=-\nabla{p}+\frac{\gamma}{Re}\Delta\vv
+\frac{1-\gamma}{Re}\nabla\cdot\sigma,\\
&\na\cdot\vv=0,\\
&\nn\times\big(\hh-\gamma_1\NN-\gamma_2\DD\cdot\nn\big)=0,
\end{split}\right.
\end{eqnarray}
where $\vv$ is the velocity of the fluid, $p$ is the pressure, $Re$ is  the Reynolds number and $\gamma\in (0,1)$.
The stress $\sigma$ is modeled by the phenomenological constitutive relation
\beno
\sigma=\sigma^L+\sigma^E,
\eeno
where $\sigma^L$ is the viscous (Leslie) stress
\begin{eqnarray}\label{eq:Leslie stress}
\sigma^L=\alpha_1(\nn\nn:\DD)\nn\nn+\alpha_2\nn\NN+\alpha_3\NN\nn+\alpha_4\DD
+\alpha_5\nn\nn\cdot\DD+\alpha_6\DD\cdot\nn\nn \een
with $\DD=\frac{1}{2}(\kappa^T+\kappa), \kappa=(\na \vv)^T$, and
\beno
\NN=\nn_t+\vv\cdot\nabla\nn+\BOm\cdot\nn,\quad\BOm=\frac{1}{2}(\kappa^T-\kappa).
\eeno
The six constants $\al_1, \cdots, \al_6$ are called the Leslie coefficients.  While, $\sigma^E$ is the elastic (Ericksen) stress
\begin{eqnarray}\label{eq:Ericksen}
\sigma^E=-\frac{\partial{E_F}}{\partial(\nabla\nn)}\cdot(\nabla\nn)^T,
\end{eqnarray}
where $E_F=E_F(\nn,\nabla\nn)$ is the Oseen-Frank energy with the form
\beno
E_F=\f {k_1} 2(\na\cdot\nn)^2+\f {k_2} 2|\nn\cdot(\na\times\nn)|^2+\f {k_3} 2|\nn\times(\na\times \nn)|^2
+\frac12{(k_2+k_4)}\big(\textrm{tr}(\na\nn)^2-(\na\cdot\nn)^2\big).
\eeno
Here $k_1, k_2, k_3, k_4$ are the elastic constant.
Especailly, in the case when $k_1=k_2=k_3=1$ and $k_4=0$, we have $E_F=\f12|\na\nn|^2$, and the molecular field $\hh$ is given by
\beno
&&\hh=-\frac{\delta{E_F}}{\delta{\nn}}=
\nabla\cdot\frac{\partial{E_F}}{\partial(\nabla\nn)}-\frac{\partial{E_F}}{\partial\nn}=-\Delta\nn,\\
&&\big(\sigma^E\big)_{ij}=-\big(\nabla\nn\odot\nabla\nn\big)_{ij}=-\partial_in_k\partial_jn_k.
\eeno
The Leslie coefficients and $\gamma_1, \gamma_2$ satisfy the following relations
\ben
&\alpha_2+\alpha_3=\alpha_6-\alpha_5,\label{Leslie relation}\\
&\gamma_1=\alpha_3-\alpha_2,\quad \gamma_2=\alpha_6-\alpha_5,\label{Leslie-coeff}
\een
where (\ref{Leslie relation}) is called Parodi's relation derived from the Onsager reciprocal relation \cite{Parodi}. These two relations
ensure that the system has a basic energy dissipation law:
\begin{align}
-\frac{\ud}{\ud{t}}\Big(\int_{\BR}\frac{Re}{2(1-\gamma)}|\vv|^2\ud\xx+E_F\Big)
=&\int_{\BR}\Big(\frac{\gamma}{1-\gamma}|\nabla\vv|^2+(\alpha_1+\frac{\gamma_2^2}{\gamma_1})(\DD:\nn\nn)^2
+\alpha_4\DD:\DD\qquad\nonumber\\
&\quad+(\alpha_5+\alpha_6-\frac{\gamma_2^2}{\gamma_1})|\DD\cdot\nn|^2
+\frac{1}{\gamma_1}|\nn\times\hh|^2\Big)\ud\xx.\quad\label{EL_energy_law}
\end{align}

\subsection{From the molecular kinetic theory to the continuum theory}

Two kinds of theories were put forward to investigate the liquid crystalline polymers from the different points of view.
The Q-tensor theory and Ericksen-Leslie theory are phenomenological.  Especially,  there are some unknown parameters in the continuum theory,
which are difficult to determine by using experimental results.
In the spirit of Hilbert sixth problem, it is very important to explore the relationship between these two theories.

Our goal is to derive two commonly used continuum theories: Q-tensor theory and Ericksen-Leslie theory starting from Doi-Onsager theory.
In \cite{KD, EZ}, Kuzzu-Doi and E-Zhang formally derive the Ericksen-Leslie
equation from the Doi-Onsager equations by taking small Deborah number limit.
We justify their formal derivation in our recent work \cite{WZZ}. An natural
question is whether one can derive the Ericksen-Leslie model from the Q-tensor
models and derive the Q-tensor model from the Doi-Onsager equations. In the
static case, similar questions have been studied by Ball-Majumdar \cite{BM},
Majumdar-Zarnescu \cite{MZ} and the second paper in this series \cite{HLWZ}.

In this paper,  we first derive a new Q-tensor model by the Bingham closure, then we derive the Ericksen-Leslie equation from the derived Q-tensor model (see Fig. 1).
The existing Q-tensor models are usually derived by various closure approximations, such as the Doi's
quadratic closure \cite{Doi}, the HL closures \cite{HL}, the orthotropic closure \cite{CL} and the Bingham closure
\cite{CT}. Feng et al. \cite{Feng} provided detailed numerical comparisons among five commonly
used closures and found that the Bingham closure gives better results than others. Moreover, in these closure methods,
the Bingham closure seems to be the only one which persists the energy dissipation law.

\begin{figure}
\centering
  \includegraphics[width=13cm]{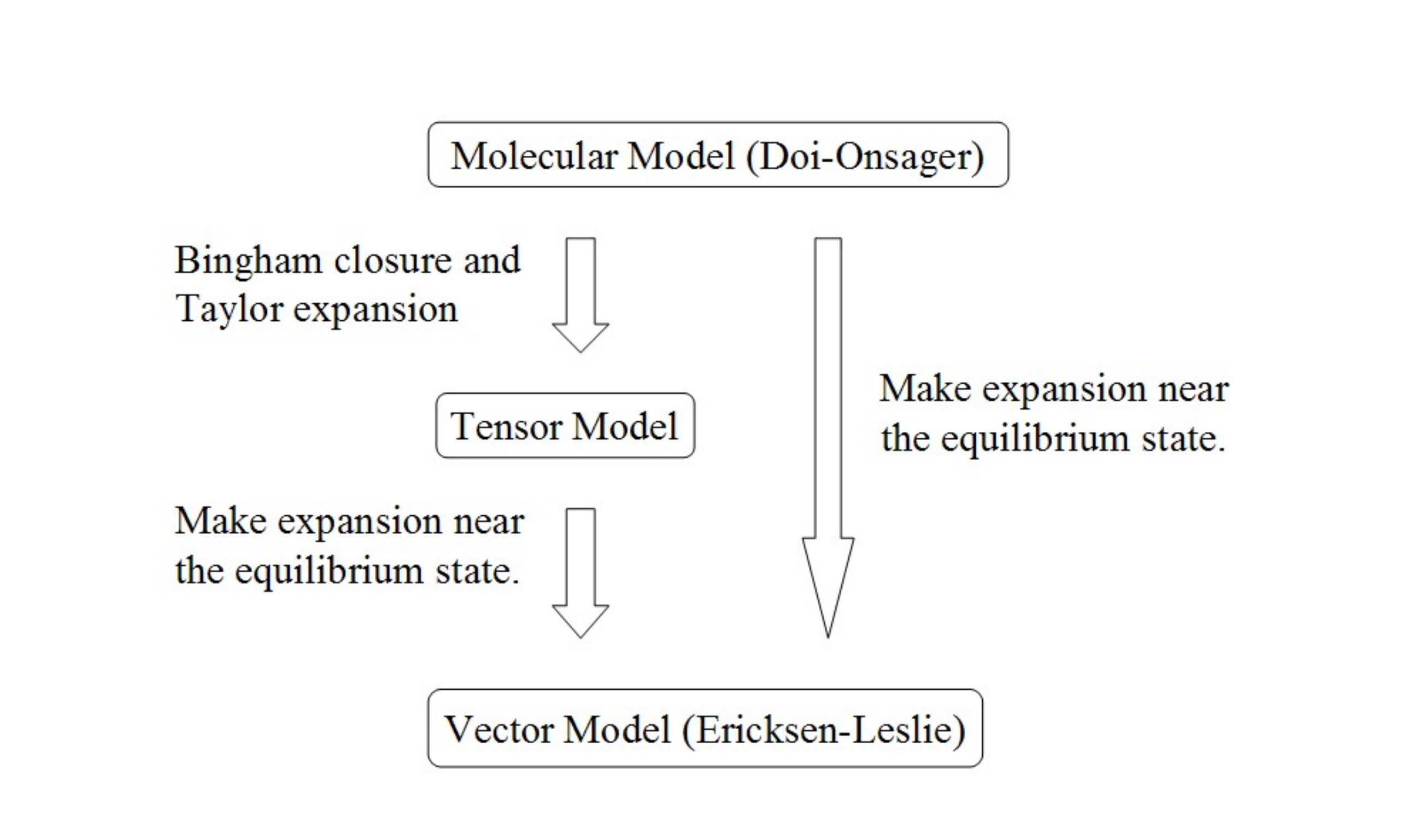}
  \caption{The systematic approach to connect the liquid crystal models of different levels.}
\end{figure}

\section{From Doi-Onsager equation to Q-tensor equation}

\subsection{Bingham closure and Q-tensor equation}

The Bingham closure is a kind of quasi-equilibrium closure approximation.
The idea is to calculate $\langle\mm\mm\mm\mm\rangle_f$ by using
the quasi-equilibrium $f_Q$ given by
\begin{align}
f_Q=\frac{\exp(\mm\mm:B_Q)}{\int_\BS\exp(\mm\mm:B_Q)\ud\mm},\quad\quad
Q=\int_\BS(\mm\mm-\frac{1}{3}\II)f\ud\mm.
\end{align}
Given a symmetric and trace free matrix $Q$, the symmetric traceless matrix $B_Q$ is determined by the following relation:
\begin{align}
Q=\frac{\int_\BS(\mm\mm-\frac{1}{3}\II)\exp({\mm\mm:B_Q})\ud\mm}{\int_\BS\exp({\mm\mm:B_Q})\ud\mm}.
\end{align}
\begin{remark}
If the eigenvalues $\lambda_i(i=1,2,3)$ of $Q$ satisfy the constraint:
\beno
-\f13<\lambda_i<\f23\quad i=1,2,3,
\eeno
then $B_Q$ is uniquely determined by $Q$(see \cite{BM}).
\end{remark}

We denote
\beno
&&Z_Q=\int_\BS\exp({\mm\mm:B_Q})\ud\mm,\quad M^{(4)}_Q=\int_\BS\mm\mm\mm\mm{f}_Q\ud\mm,\\
&&M^{(6)}_Q=\int_\BS\mm\mm\mm\mm\mm\mm{f}_Q\ud\mm,\quad Q^\ve=\int_\BR g_\ve(\xx-\xx')Q(\xx')\ud\xx'.
\eeno
It is easy to compute that
\begin{align*}
U_\ve{f}&=-\alpha\int_{\BR}\int_\BS(\mm\cdot\mm')^2f(\xx',\mm')g_\ve(\xx-\xx')\ud\xx'\ud\mm\\
&=-\alpha(\mm\mm-\frac{1}{3}\II):Q^\ve,\\
\mu_\ve{f}&=\mm\mm:B_Q-\ln{Z_Q}-\alpha(\mm\mm-\frac{1}{3}\II):Q^\ve,\\
\CR\mu_\ve{f}&=\mm\times\big((B_Q-\alpha{Q}^\ve)\cdot\mm\big).
\end{align*}
Introduce the operator
\begin{align*}
\mathcal{M}_{Q}(A)&=\frac13A+Q\cdot A-A:M_Q^{(4)},\\
\mathcal{N}_{Q}(A)_{\alpha\beta}&=\partial_i\bigg\{\Big[\gamma_\bot\Big(M^{(4)}_{Q\alpha\beta kl}\delta_{ij}
-\frac{\delta_{\alpha\beta}}{3}Q_{kl}\delta_{ij}\Big)
+(\gamma_{\|}-\gamma_\bot)\big(M^{6}_{\alpha\beta klij}
-\frac{\delta_{\alpha\beta}}{3}M^{(4)}_{klij}\big)\Big]\partial_jA_{kl}\bigg\}.
\end{align*}
Then the Doi-Onsager equation (\ref{Doi-v}) is transformed to a system in terms of $(Q,\vv)$:
\begin{align}
\frac{\pa{{Q}}}{\pa{t}}+\vv\cdot\nabla{Q}&=\frac{\ve}{De}\mathcal{N}_Q(B_{Q}-\alpha Q^\ve)+\frac{1}{De}
\Big(-6Q+2\alpha\big[\CM_Q(Q^\ve)+\CM_Q^T(Q^\ve)\big]\Big)\nonumber\\
&\quad+\big(\CM_Q(\kappa^T)+\CM_Q^T(\kappa^T)\big),\label{eq:Q-I}\\\nonumber
\frac{\pa{\vv}}{\pa{t}}+\vv\cdot\nabla\vv&=-\nabla{p}
+\frac{\gamma}{Re}\Delta\vv+\frac{1-\gamma}{2Re}\nabla\cdot(\DD:M_Q^{(4)})\\
&\quad+\frac{1-\gamma}{DeRe}\Big(-\nabla\cdot\big[-3Q+2\alpha\CM_Q(Q^\ve)\big]+\alpha{Q}:\nabla{Q}^\ve\Big).\label{eq:v-I}
\end{align}
Since the typical interaction distance $\sqrt{\ve}$ is very small, we make the following Taylor expansion for $Q^\ve$:
\begin{align*}
Q^\ve(\xx)&=Q(\xx)+\ve\frac{G}{2}\Delta{Q}(\xx)+O(\ve^2),
\end{align*}
where the constant $G=\frac13\int_{\BR}g(\yy)|\yy|^2\ud\yy.$

Replacing $Q^\ve$ by $Q(x)+\ve\frac{G}{2}\Delta{Q}(x)$ in (\ref{eq:Q-I})-(\ref{eq:v-I}), we derive the following Q-tensor equation
from the Doi-Onsager equation:
\begin{align}
\frac{\pa{{Q}}}{\pa{t}}+\vv\cdot\nabla{Q}&=\frac{\ve}{De}\mathcal{N}(B_Q-\alpha Q
-\frac{G}{2}\alpha\ve\Delta Q)+\big(\CM_Q(\kappa^T)+\CM_Q^T(\kappa^T)\big)\nonumber\\
&\quad+\frac{1}{De}\Big(-6Q+2\alpha\big[\CM_Q(Q+\frac{G}{2}\ve\Delta{Q})
+\CM_Q^T(Q+\frac{G}{2}\ve\Delta{Q})\big]\Big),\label{eq:Q-D}\\\nonumber
\frac{\pa{\vv}}{\pa{t}}+\vv\cdot\nabla\vv&=-\nabla{p}
+\frac{\gamma}{Re}\Delta\vv+\frac{1-\gamma}{2Re}\nabla\cdot(\DD:M_Q^{(4)})\\
&\quad+\frac{1-\gamma}{DeRe}\Big(-\nabla\cdot\big[-3Q+2\alpha\CM_Q(Q+
\frac{G}{2}\ve\Delta{Q})\big]+\frac12G\alpha\ve{Q}:\nabla\Delta{Q}\Big).\label{eq:v-D}
\end{align}

\subsection{Energy dissipation law of Q-tensor equation}

In this subsection, we derive the energy dissipation law of Q-tensor equation (\ref{eq:Q-D})-(\ref{eq:v-D}).
We need the following lemmas.

\begin{lemma}\label{eq:ident-1}
For any symmetric and trace free matrix $Q$, there hold
\begin{align}
&\frac32Q=\CM_Q(B_Q)\equiv B_Q\cdot{Q}+\frac{1}{3}{B_Q}-B_Q:M_Q^{(4)},\nonumber\\
&B_Q\cdot{Q}=Q\cdot{B_Q}.\nonumber
\end{align}
\end{lemma}
\begin{proof} For any constant matrix $\kappa$, we have
\begin{align*}
0&=\int\Big(\mm\times(\kappa\cdot\mm)\cdot\CR{f_Q}+
\CR\cdot\big(\mm\times(\kappa\cdot\mm)\big)f_Q\Big)\ud\mm\nonumber\\
&=\int\Big(2\big(\mm\times(\kappa\cdot\mm)\big)\cdot\big(\mm\times({B_Q}\cdot\mm)\big){f_Q}+
(\II-3\mm\mm):\kappa f_Q\Big)\ud\mm\\
&=\int\bigg\{2(\kappa\cdot\mm)(B_Q\cdot\mm)f_Q-2(\kappa:\mm\mm)(B_Q:\mm\mm)f_Q+
(\II-3\mm\mm):\kappa f_Q\bigg\}\ud\mm\\
&=\kappa:\big(2B_Q\cdot\MM_Q-2B_Q:M_Q^{(4)}-3Q\big),
\end{align*}
where $\MM_Q=\int_\BS\mm\mm f_Q\ud\mm=Q+\f13 I$,
which implies the first equality. The second equality is a direct consequence of the first equality by noting that $Q, B_Q$ are symmetric.
\end{proof}

\begin{lemma}\label{lem:ident-2}
For any matrix $E$, there holds
\begin{align}
M_Q(E):E&\equiv E:(Q+\frac{1}{3}\II)\cdot{E}-E:M_Q^{(4)}:E=\int_\BS|\mm\times(E\cdot\mm)|^2f_Q\ud\mm\ge0.\nonumber\\
\int\mathcal{N}(E):E\ud{\xx}&=-\int\big[\gamma_{\|}\mm\mm+\gamma_{\bot}(\II-\mm\mm)\big]_{ij}f_Q
\partial_i\big[(\mm\mm-\frac\II3):E\big]\partial_j\big[(\mm\mm-\frac\II3):E\big]\ud{x}\le0.\nonumber
\end{align}
\end{lemma}

We define the energy functional
\begin{align}
E_1(Q)=\int_{\BR}\Big(-\ln{Z}_Q+Q:B_Q-
\frac{\alpha}{2}{Q}:{Q}^\ve\Big)\ud\xx.
\end{align}
The system (\ref{eq:Q-I})-(\ref{eq:v-I}) obeys the energy law
\begin{align}
&\frac{\ud}{\ud{t}}\bigg\{\int_{\BR}\frac{1}{2}|\vv|^2\ud{x}+\frac{1-\gamma}{ReDe}E_1(Q)\bigg\}\nonumber\\
&=-\int_{\BR}\bigg\{\frac{\gamma}{Re}|\DD|^2+\frac{1-\gamma}{2Re}\DD:M_Q^{(4)}:\DD -(B_Q-\alpha Q^\ve):\mathcal{N}(B_Q-\alpha Q^\ve)\nonumber\\
&\qquad+\frac{4(1-\gamma)}{ReDe^2}\big(B_Q-\alpha{Q}^\ve\big):\CM_Q\big(B_Q-\alpha{Q}^\ve\big)\bigg\}\ud{x}.\label{eq:energy-1}
\end{align}
An important property is that the energy law (\ref{eq:energy-1}) is dissipated by Lemma \ref{lem:ident-2}.

Now let us prove (\ref{eq:energy-1}). It is easy to see that
\begin{align}
\CM_Q(A):B=\CM_Q(B):A,\label{eq:M_Q}
\end{align}
if $A$ or $B$ is symmetric. A direct computation tells us that
\begin{align}
\frac{\partial}{\partial Q} \Big(-\ln{Z}_Q+ Q:B_Q\Big)=B_Q.\label{eq:lnZ}
\end{align}
Then by using (\ref{eq:Q-I})-(\ref{eq:v-I}), we obtain
\begin{align*}
&\frac{\ud}{\ud{t}}E_1(Q)=\int_{\BR}\big(B_Q-\alpha{Q}^\ve\big):\partial_tQ\ud{x}\\
&=\int_{\BR}\frac{1}{De}\big(B_Q-\alpha{Q}^\ve\big):
\Big(-6Q+2\alpha\big[\CM_{Q}(Q^\ve)+\CM_{Q}^T(Q^\ve)\big]\Big)+\frac{\ve}{De}(B_Q-\alpha Q^\ve):\mathcal{N}(B_Q-\alpha Q^\ve)\\
&\qquad+\big(B_Q-\alpha{Q}^\ve\big):\Big(\CM_{Q}(\kappa^T)+\CM_{Q}^T(\kappa^T)-\vv\cdot\nabla{Q}\Big)\ud{x},\\
&\frac{\ud}{\ud{t}}\int_{\BR}\frac{1}{2}|\vv|^2\ud{x}=\int_{\BR}\Big\{-\frac{\gamma}{Re}|\DD|^2-\frac{1-\gamma}{2Re}\DD:M_Q^{(4)}:\DD\\
&\qquad\qquad+\frac{1-\gamma}{DeRe}\Big(\big[-3Q+2\alpha\CM_{Q}(Q^\ve)\big]:\nabla\vv
+\alpha{Q}:(\vv\cdot\nabla){Q}^\ve\Big)\Big\}\ud{x},
\end{align*}
from which, Lemma \ref{eq:ident-1} and (\ref{eq:M_Q}), we infer that
\begin{align*}
&\frac{\ud}{\ud{t}}\Big(\frac{1}{ReDe}E_1(Q)+\int_\BR\frac{1}{2(1-\gamma)}|\vv|^2\ud{x}\Big)\\
&=\int_\BR\bigg\{\frac{\ve}{ReDe^2}(B_Q-\alpha Q^\ve):\mathcal{N}(B_Q-\alpha Q^\ve)-\frac{4}{ReDe^2}\big(B_Q-\alpha{Q}^\ve\big):\CM_{Q}(B_Q-\alpha Q^\ve)\\
&\qquad+\frac{1}{ReDe}\big(B_Q-\alpha{Q}^\ve\big):\Big(2\CM_{Q}(\kappa^T)-\vv\cdot\nabla{Q}\Big)-\frac{\gamma}{Re(1-\gamma)}|\DD|^2
-\frac{1}{2Re}\DD:M_Q^{(4)}:\DD\\
&\qquad+\frac{1}{DeRe}\Big(-2\CM_{Q}(B_Q-\alpha Q^\ve):\nabla\vv
+\alpha{Q}:(\vv\cdot\nabla){Q}^\ve\Big)\bigg\}\ud{x}\nonumber\\
&=-\int_\BR\bigg\{-\frac{\ve}{ReDe^2}(B_Q-\alpha Q^\ve):\mathcal{N}(B_Q-\alpha Q^\ve)
+\frac{4}{ReDe^2}\big(B_Q-\alpha{Q}^\ve\big):\CM_{Q}(B_Q-\alpha Q^\ve)\\
&\qquad+\frac{\gamma}{Re(1-\gamma)}|\DD|^2
+\frac{1}{2Re}\DD:M_Q^{(4)}:\DD\bigg\}\ud{x}.
\end{align*}

Define the energy functional
\begin{align}
E_2(Q)=\int_\BR-\ln{Z}_Q+ Q:B_Q+\frac{\alpha}{2}\big(-|{Q}|^2+\frac{G}{2}\ve|\nabla{Q}|^2\big)\ud{x}.
\end{align}
Then the system (\ref{eq:Q-D})-(\ref{eq:v-D}) obeys the following energy dissipation law
\begin{align}
&\frac{\ud}{\ud{t}}\Big\{\int_\BR\frac{1}{2}|\vv|^2\ud{x}+\frac{1-\gamma}{ReDe}E_2(Q)\Big\}
=-\int_\BR\bigg\{\frac{\gamma}{Re}|\DD|^2+\frac{1-\gamma}{2Re}\DD:M_Q^{(4)}:\DD\nonumber\\
&\qquad-\frac{\ve(1-\gamma)}{ReDe^2}\big(B_Q-\alpha{Q}-\frac{G}{2}\alpha\ve\Delta{Q}\big)
:\mathcal{N}\big(B_Q-\alpha{Q}-\frac{G}{2}\alpha\ve\Delta{Q}\big)\nonumber\\\label{eq:energy-2}
&\qquad+\frac{4(1-\gamma)}{ReDe^2}\big(B_Q-\alpha{Q}-\frac{G}{2}\alpha\ve\Delta{Q}\big):
\CM_Q\big(B_Q-\alpha{Q}-\frac{G}{2}\alpha\ve\Delta{Q}\big)\bigg\}\ud{x}.
\end{align}

\subsection{Some remarks on new Q-tensor equation}

New tensor equation (\ref{eq:Q-D})-(\ref{eq:v-D}) derived from the Doi-Onsager equation by the Bingham closure  keeps many important physical properties:

\begin{itemize}

\item Two kinds of physical diffusions(translational and rotational diffusion) are preserved;

\item The parameters are not phenomenological but have definite physical meaning;

\item The eigenvalues of $Q$  satisfy the physical constrain: $-\frac13\le\lambda_i\le\frac23$ if they are satisfied initially;

\item The Ericksen-Leslie equation can be deduced from the Q-tensor equation by taking the small Deborah number limit.

\end{itemize}

Furthermore, we want to comment that Q-tensor equation can be viewed as some kind of regularization of the Ericksen-Leslie equation, since it removes
the constrain $|\nn|=1$, and allows the liquid crystal to be biaxial. Motivated by the work of the harmonic heat flow,
the classical regularization of the Ericksen-Leslie equation is to add
a penality term $\f1 {\ve^2}(|\nn|^2-1)^2$ in $E_F$ to remove some higher-order nonlinearities due to the constrain $|\nn|=1$ \cite{LL}.
That is so called the Ginzburg-Landau approximation. However, whether the Ericksen-Leslie equation
can be recovered from the approximated equation by taking $\ve\rightarrow 0$ is still a challenging question. Moreover, the physical meaning of Ginzburg-Landau approximation is also unclear.

\section{The critical points of the bulk free energy}

Let $\CF_{B}(Q)$ be the bulk energy of $E_2(Q)$, that is,
\begin{align}
\CF_{B}(Q)=\int_{\BR}-\ln{Z}_Q+ Q:B_Q-\frac{\alpha}{2}|{Q}|^2 \ud\xx.
\end{align}
Due to (\ref{eq:lnZ}), it is easy to see that  its critical point satisfies the  equation
$$B_Q=\alpha Q.$$
The solution of this equation is related to the critical points
of the Maier-Saupe energy functional. More precisely, the following proposition has been proven by \cite{FS, LZZ, ZWFW}.

\begin{proposition}\label{prop:q-1}
Let  $\eta$ be a solution of the equation
\begin{eqnarray}\label{eta-alpha}
\frac{3\ue^\eta}{\int_0^1\ue^{\eta{z^2}}\ud{z}}=3+2\eta+\frac{\eta^2}{\alpha}.
\end{eqnarray}
Then
\begin{align}
B_Q-\alpha{Q}=0 \quad\Longleftrightarrow\quad B_Q=\eta(\nn\nn-\frac{1}{3}\II),\quad \nn\in \BS.
\end{align}
There exists a critical number $\alpha^*>0$ such that
\begin{enumerate}
\item When $\alpha<\alpha^*$, $\eta=0$ is the only solution of (\ref{eta-alpha});
\item When $\alpha=\alpha^*$, except  $\eta=0$, there is another solution $\eta=\eta^*$ of (\ref{eta-alpha});
\item When $\alpha>\alpha^*$, except $\eta=0$, there are other two solutions $\eta_1>\eta^*>\eta_2$ of (\ref{eta-alpha}).
\end{enumerate}
\end{proposition}

In the sequel, we always take $\eta=\eta_1$. Let $A_k=\int_0^1 x^k{\ue^{\eta x^2}}\ud x.$  The following  facts have been proved in  \cite{WZZ}:  for $\eta>\eta^*$
\begin{align}\label{fact-1}
{3A_2^2+2A_0A_2-5A_0A_4}>0,\quad 6A_2-5A_4-A_0>0.
\end{align}
We also define
\begin{align}\label{order-parameter}
S_2=\frac{3A_2-A_0}{2A_0},\quad S_4=\frac{1}{8A_0}(35A_4-30A_2+3A_0).
\end{align}
If  $B_{Q_0}=\alpha Q_0$, then we have
\ben
Q_0=S_2(\nn\nn-\frac13\II),\quad B_0\equiv B_{Q_0}=\eta(\nn\nn-\frac13\II).
\een
In such case, we have
\beno
f_0&\equiv& f_{Q_0}=\frac{\ue^{\eta(\mm\cdot\nn)^2}}{\int_\BS\ue^{\eta(\mm\cdot\nn)^2}\ud\mm},\\
M_{Q_0,{ijkl}}^{(4)}&=&S_4n_in_jn_kn_l+\frac{S_2-S_4}{7}\Big(n_in_j\delta_{kl}+n_in_k\delta_{jl}+n_in_l\delta_{jk}+n_jn_k\delta_{il}\nonumber\\
&&+n_jn_l\delta_{ik}+n_kn_l\delta_{ij}\Big)+\big(\frac{S_4}{35}-\frac{2S_2}{21}+\frac{1}{15}\big)(\delta_{ij}\delta_{kl}+\delta_{ik}\delta_{jl}
+\delta_{il}\delta_{jk}).
\eeno

Now we introduce some linear operators associated to the critical point $Q_0$,  which will play an important role in the derivation of the Ericksen-Leslie
equation  from the Q-tensor equation. We first introduce the space
\begin{align}
\mathfrak{S}=\{B\in\mathbb{M}^{3\times3}: B \text{ is symmetric and trace free}\}\nonumber
\end{align}
endowed with the inner product
\begin{align}
\langle B_1, B_2\rangle_{\FS}\equiv B_1:B_2=\sum_{i,j=1,2,3}B_{1ij}B_{2ij}.\nonumber
\end{align}
Given  $\nn\in \BS$,  the linear operators $\mathcal{Q}_{\nn}, \mathcal{J}_{\nn},  \mathcal{K}_{\nn}, \mathcal{L}_{\nn}$: $\mathfrak{S}\mapsto\mathfrak{S}$ defined as
\begin{align*}
&\mathcal{Q}_{\nn}(B)=M_{Q_0}^{(4)}:B-\frac{1}{3}\II(Q_0:B)-Q_0(Q_0:B),\\
&\mathcal{J}_{\nn}(B)=\frac{1}{3}B+\frac{1}{2}(B\cdot Q_0+Q_0\cdot B)-B:{M^{(4)}_{Q_0}},\\
&\mathcal{K}_{\nn}(B)=B-\alpha\CQ_\nn(B),\\
&\mathcal{L}_{\nn}(B)=-2\mathcal{J_\nn}(\mathcal{K_\nn}(B)),
\end{align*}
where the operator $\CL_\nn$ can be written as
\begin{align*}
\CL_\nn(B)=\frac{1}{3}(B-\alpha\CQ_\nn(B))+\frac{1}{2}\big((B-\alpha\CQ_\nn(B))\cdot Q_0+Q_0\cdot(B-\alpha\CQ_\nn(B))\big)-(B-\alpha\CQ_\nn(B)):{M^{(4)}_{Q_0}}.
\end{align*}

\begin{proposition}We have the following properties:
\begin{itemize}
\item[1.] $\mathcal{Q}_\nn$, $\mathcal{J}_\nn$, $\mathcal{K}_\nn,  \mathcal{L}_{\nn}$ are self-adjoint;
\item[2.] $\mathcal{J}_\nn$ is positive, and $\mathcal{K}_\nn$ is non-negative;
\item[3.] $\mathrm{Ker}~\mathcal{K}_{\nn}=\mathrm{Ker}~\mathcal{L}_{\nn}=\big\{\nn\nn^\bot+\nn^\bot\nn: \nn^\bot\in \BR, \nn\cdot \nn^\bot=0\big\}$.
\end{itemize}
\end{proposition}

\begin{proof} It is easy to see that  $\mathcal{Q}_\nn$, $\mathcal{J}_\nn$, $\mathcal{K}_\nn$ are self-adjoint.
In the following, we prove that   $\mathcal{L}_{\nn}$ is self-adjoint. We have
\begin{align}
Q_0\cdot\CQ_\nn(B_1)=&Q_0\cdot\Big(S_4\nn\nn\nn\nn:B_1+\frac{S_2-S_4}{7}\big(
2(\nn\nn\cdot B_1+B_1\cdot\nn\nn)+\II\nn\nn:B_1\big)\nonumber\\
&+\big(\frac{S_4}{35}-\frac{2S_2}{21}
+\frac{1}{15}\big)2B_1\Big)-\frac13Q_0\cdot(\nn\nn:B_1)\nonumber\\
=&S_2S_4\frac{2}{3}\nn\nn\nn\nn:B_1+\frac{S_2(S_2-S_4)}{7}\Big(\frac43\nn\nn\cdot B_1+3\nn\nn(\nn\nn:B_1)
-\frac23B_1\cdot\nn\nn-\frac13\II(\nn\nn:B_1)\Big)\nonumber\\
&+\big(\frac{S_4}{35}-\frac{2S_2}{21}
+\frac{1}{15}\big)2S_2\big(\nn\nn\cdot B_1-\frac13B_1\big)-\frac{S_2^2}3(\nn\nn-\frac13\II)(\nn\nn:B_1).\nonumber
\end{align}
Hence,
\begin{align*}
\langle Q_0\cdot\CQ_\nn(B_1), B_2\rangle=\langle Q_0\cdot\CQ_\nn(B_2), B_1\rangle.
\end{align*}
Similarly, we have
\begin{align*}
\langle \CQ_\nn(B_1)\cdot Q_0, B_2\rangle=\langle \CQ_\nn(B_2)\cdot Q_0, B_1\rangle.
\end{align*}
On the other hand, we have
\begin{align*}
&\big<{\CQ_\nn(B_1)}:M_{Q_0}^{(4)}, B_2\big>=\big<{\CQ_\nn(B_1)}, M_{Q_0}^{(4)}:B_2\big>\\
&=\big<M_{Q_0}^{(4)}:B_1+\frac{S_2}{3}\II(\nn\nn:B_1)+
S_2^2(\nn\nn-\frac13\II)(\nn\nn:B_1), M_{Q_0}^{(4)}:B_2\big>\\
&=\big<M_{Q_0}^{(4)}:B_1, M_{Q_0}^{(4)}:B_2\big>+S_2^2\Big((\nn\nn:B_1)\big[\nn\nn:M_{Q_0}^{(4)}:B_2\big]\Big)\\
&=\big<M_{Q_0}^{(4)}:B_1, M_{Q_0}^{(4)}:B_2\big>+S_2^2(\nn\nn:B_1)(\nn\nn:B_2)\big(\frac{12S_4}{35}+\frac{11S_2}{21}+\frac{1}{15}\big)\\
&=\big<{\CQ_\nn(B_2)}:M_{Q_0}^{(4)}, B_1\big>.
\end{align*}
Thus,  the operator $\CL_\nn$ is self adjoint.

It follows from Lemma \ref{lem:ident-2} that  the operator $\CJ_\nn$ is positive.  Now we prove that the operator $\CK_\nn$ is non-negative.  By the definition, we have
\begin{align}
\langle\CK_\nn(B), B\rangle=&|B|^2-
\alpha M_{Q_0}^{(4)}:BB+\alpha(Q_0:B)^2\nonumber\\
=&|B|^2+\alpha\Big( (S_2^2-S_4)(\nn\nn:B)^2-\frac{S_2-S_4}{7}4|\nn\cdot B|^2
-\big(\frac{S_4}{35}-\frac{2S_2}{21}
+\frac{1}{15}\big)2|B|^2\Big).\nonumber
\end{align}
Thanks to the definition of $S_2$ and $S_4$,  we know that
\begin{align*}
&\frac{S_2-S_4}{7}=\frac{1}{8A_0}(-5A_4+6A_2-A_0),\\
&\frac{S_4}{35}-\frac{2S_2}{21}+\frac{1}{15}=\frac{1}{8A_0}(A_4-2A_2+A_0).
\end{align*}
Choose an orthogonal matrix $P=(\nn_1, \nn_2, \nn_3)^T$ such that $\nn_3=\nn$,
hence $P\cdot\nn=\mathbf{e}_3=(0,0,1)^T$. Let $\hat{B}=PBP^T$. We have
\begin{align*}
&|B|^2=|\hat{B}|^2,\\
&|B\cdot\nn|^2=|PB\cdot\nn|^2=|\hat{B}\cdot\mathbf{e}_3|^2=\hat{B}_{13}^2+\hat{B}_{23}^2+\hat{B}_{33}^2,\\
&\nn\nn:B=\nn^TB\nn=\mathbf{e}_3^T\cdot PBP^T\cdot\mathbf{e}_3=\hat{B}_{33}.
\end{align*}
Then we get
\begin{align}
\langle\CK_\nn(B), B\rangle
=&\hat{B}_{ij}^2+\alpha\Big( (S_2^2-S_4)\hat{B}_{33}^2-\frac{S_2-S_4}{7}4(\hat{B}_{13}^2+\hat{B}^2_{23}+\hat{B}^2_{33})
-\big(\frac{S_4}{35}-\frac{2S_2}{21}
+\frac{1}{15}\big)2\hat{B}_{ij}^2\Big)\nonumber\\
=&2\Big(1-2\alpha\big(\frac{S_4}{35}-\frac{2S_2}{21}
+\frac{1}{15}\big)\Big)\hat{B}_{12}^2+\Big(1-2\alpha\big(\frac{S_4}{35}-\frac{2S_2}{21}
+\frac{1}{15}\big)\Big)(\hat{B}_{11}^2+\hat{B}_{22}^2)\nonumber\\
&+\Big(2-4\alpha\big(\frac{S_4}{35}-\frac{2S_2}{21}
+\frac{1}{15}\big)-4\alpha\frac{S_2-S_4}{7}\Big)(\hat{B}_{13}^2+\hat{B}_{23}^2)\nonumber\\
&+\Big(1-2\alpha\big(\frac{S_4}{35}-\frac{2S_2}{21}
+\frac{1}{15}\big)+\alpha(S_2^2-S_4)-4\alpha\frac{S_2-S_4}{7}\Big)\hat{B}_{33}^2\nonumber\\
=&\frac{6A_2-5A_4-A_0}{2(A_2-A_4)}\hat{B}_{12}^2+\frac{6A_2-5A_4-A_0}{4(A_2-A_4)}(\hat{B}_{11}^2+\hat{B}_{22}^2)\nonumber\\
&+\Big(-\frac{6A_2-5A_4-A_0}{4(A_2-A_4)}+\alpha(S_2^2-S_4)\Big)(\hat{B}_{11}+\hat{B}_{22})^2\nonumber\\
=&\frac{6A_2-5A_4-A_0}{2(A_2-A_4)}\hat{B}_{12}^2
-\frac{6A_2-5A_4-A_0}{2(A_2-A_4)}\hat{B}_{11}\hat{B}_{22}+\alpha(S_2^2-S_4)(\hat{B}_{11}+\hat{B}_{22})^2.\nonumber
\end{align}
Furthermore, we know by (\ref{fact-1}) that
\begin{align}
&4\alpha(S_2^2-S_4)-\frac{6A_2-5A_4-A_0}{2(A_2-A_4)}\nonumber\\
&=\frac{4A_0}{A_2-A_4}\cdot\frac{2(3A_2-A_0)^2-A_0(35A_4-30A_2+3A_0)}{8A_0^2}
-\frac{6A_2-5A_4-A_0}{2(A_2-A_4)}\nonumber\\
&=\frac{1}{A_2-A_4}\cdot\frac{2(3A_2-A_0)^2-2A_0(15A_4-12A_2+A_0)}{2A_0}\nonumber\\
&=\frac{3}{A_2-A_4}\cdot\frac{3A_2^2+2A_0A_2-5A_0A_4}{A_0}>0,\nonumber
\end{align}
and $6A_2-5A_4-A_0>0.$  This proves
$$\langle\CK_\nn(B), B\rangle\ge0.$$
Moreover, the equality holds if and only if
$\hat B_{ij}=0$ for $\{i,j\}\neq\{1,3\}, \{2,3\}$,
which means that
$$B=P^T\hat{B}P=\hat{B}_{13}\nn\otimes\nn_1^T+\hat{B}_{23}\nn\otimes\nn_2^T
+\hat{B}_{13}\nn_1\otimes\nn^T+\hat{B}_{23}\nn_2\otimes\nn^T=\nn\nn^\bot+\nn^\bot\nn,$$
where $\nn^\bot=\hat{B}_{13}\nn_1+\hat{B}_{23}\nn_2$ satisfies $\nn^\bot\cdot\nn=0$.
This proves the second point and third point of the proposition.
\end{proof}

\begin{proposition}\label{prop:q6-linear}
We have
\begin{align}
B_0:\big[M^{(6)}_{Q_0}:B-{M^{(4)}_{Q_0}}(Q_0:B)\big]=B_0\cdot\CQ_\nn(B)+B\cdot{Q_0}+\frac13B-B:{M^{(4)}_{Q_0}}-\frac32\CQ_\nn(B).\nonumber
\end{align}
\end{proposition}

\begin{proof} Proposition follows from the following two equalities:
\begin{align*}
&\int_\BS\big((\mm\cdot\vv)(\mm\times\uu)\cdot\CR\big({f_0}\mm\mm:B\big)+
\CR\cdot\big(\mm\times\uu\mm\cdot\vv)f_0\mm\mm:B\big)\ud\mm=0,\\
&\int_\BS\big((\mm\cdot\vv)(\mm\times\uu)\cdot\CR{f_0}+
\CR\cdot\big(\mm\times\uu\mm\cdot\vv)f_0\big)\ud\mm=0,
\end{align*}
for any $\uu, \vv\in \BR$.
\end{proof}
\begin{remark}
A direct consequence of Proposition \ref{prop:q6-linear} is that
$$B_0\cdot\CQ_\nn(B)+B\cdot{Q_0}=\CQ_\nn(B)\cdot B_0+{Q_0}\cdot B,$$
or equivalently(recalling $B_{0}=\al Q_0$)
\begin{align}\label{eq:conseq-symm}
(B-\alpha\CQ_\nn(B))\cdot Q_0={Q_0}\cdot(B-\alpha\CQ_\nn(B)).
\end{align}
\end{remark}

\section{From Q-tensor equation to Ericksen-Leslie equation}

In this section, we will derive the Ericksen-Leslie equation (\ref{eq:EL}) from the Q-tensor equation (\ref{eq:Q-D})-(\ref{eq:v-D})
by taking small Deborah number limit.
For this end, we take $De=\ve$.

\subsection{Formal expansion}

We make a formal Taylor expansion in $\ve$ for $(Q,\vv, B_Q)$:
\begin{align*}
&Q=Q_0+\ve{Q}_1+\ve^2Q_2+\cdots,\\
&\vv=\vv_0+\ve\vv_1+\ve^2\vv_2+\cdots,\\
&B_Q=B_0+\ve B_1+\ve^2 B_2+\cdots.
\end{align*}
Hence,
\begin{align*}
Z_Q&=\int_\BS\ue^{\mm\mm:(B_0+\ve{B_1}+\ve^2B_2+\cdots)}\ud\mm=Z_{Q_0}\big(1+\ve{Q_0:B_1}+O(\ve^2)\big),\\
Q_0+\ve{Q_1}+\cdots&=\frac{\int_\BS(\mm\mm-\frac{1}{3}\II)
\ue^{\mm\mm:B_0}\Big(1+\ve\mm\mm:B_1+O(\ve^2)\Big)\ud\mm}
{Z_{Q_0}\Big(1+\ve{Q_0:B_1}+O(\ve^2)\Big)}\\
&=Q_0+\ve\big((M_{Q_0}^{(4)}-\frac13\II{Q_0}):B_1-Q_0(Q_0:B_1)\big)+O(\ve^2).
\end{align*}
This implies that
\beno
Q_1=\big(M_{Q_0}^{(4)}-\frac13\II{Q_0}\big):B_1-Q_0(Q_0:B_1).
\eeno
We also have
\begin{align*}
M^{(4)}_Q&=\f {\int_\BS\mm\mm\mm\mm\ue^{\mm\mm:B_0}(1+\ve\mm\mm:B_1+\cdots)\ud\mm} {Z_{Q_0}\Big(1+\ve{Q_0:B_1}+O(\ve^2)\Big)}\\
&=M^{(4)}_{Q_0}+\ve\big({M^{(6)}_{Q_0}}:B_1-M^{(4)}_{Q_0}(Q_0:B_1)\big)+O(\ve^2).
\end{align*}

Plugging them into (\ref{eq:Q-D})-(\ref{eq:v-D}), we find that the terms of $O(\ve^{-1})$ satisfy
\begin{align}\label{expansion-0}
-3Q_0+2\alpha\Big(\frac{1}{3}Q_0+{Q_0}\cdot{Q_0}-{Q_0}:M_{Q_0}^{(4)}\Big)=0.
\end{align}
While the terms of $O(1)$ satisfy
\begin{align}
&\frac{\pa{{Q_0}}}{\pa{t}}+\vv_0\cdot\nabla{Q_0}=
\Big\{-6Q_1+2\alpha\big[2\CJ_\nn(Q_1)-2Q_0:\big\{{M^{(6)}_{Q_0}}:B_1
-{M^{(4)}_{Q_0}}(Q_0:B_1)\big\}\big]\Big\}\nonumber\\\label{tensor-q-expan1}
&\qquad\qquad\qquad+2G\alpha\CJ_\nn(\Delta Q_0)+\big(\CM_{Q_0}(\kappa^T)+\CM_{Q_0}^T(\kappa^T)\big),\\\nonumber
&\frac{\pa{\vv_0}}{\pa{t}}+\vv_0\cdot\nabla\vv_0=-\na p_0+
\frac{\gamma}{Re}\Delta\vv_0+\frac{1-\gamma}{2Re}\nabla\cdot(\DD_0:M_{Q_0}^{(4)})
+\frac{1-\gamma}{Re}\bigg\{-\nabla\cdot\Big[-3Q_1+2\alpha\big(\CJ_\nn(Q_1)\nonumber\\
&\quad+{G}\CM_{Q_0}(\Delta{Q_0})
-Q_0:\big\{{M^{(6)}_{Q_0}}:B_1-{M^{(4)}_{Q_0}}(Q_0:B_1)\big\}\big)\Big]
+\alpha \frac{G}{2}{Q_0}:\nabla\Delta{Q_0}\bigg\}.\label{tensor-v-expan1}
\end{align}

By Lemma \ref{eq:ident-1}, (\ref{expansion-0}) implies that
\begin{align}
\CM_{Q_0}(B_{0})-\alpha \CM_{Q_0}(Q_0)=0,\nonumber
\end{align}
which means that
\begin{align}
\CM_{Q_0}(B_{0}-\alpha Q_0)=0,\nonumber
\end{align}
or
\begin{align}
\langle\CM_{Q_0}(B_{0}-\alpha Q_0), B_{0}-\alpha Q_0\rangle=0,\nonumber
\end{align}
from which and Lemma \ref{lem:ident-2}, it follows that
\begin{align}
 B_{0}-\alpha Q_0=0.\nonumber
\end{align}
Hence, we get by Proposition \ref{prop:q-1} that
\begin{align}
B_{0}=\eta(\nn\nn-\frac13\II),\quad {Q_0}=S_2(\nn\nn-\frac13\II).
\end{align}

It follows from Proposition \ref{prop:q6-linear} that
\begin{align*}
&-6Q_1+2\alpha\Big[2\CJ_\nn(Q_1)
-2Q_0:\big\{{M^{(6)}_{Q_0}}:B_1-{M^{(4)}_{Q_0}}(Q_0:B_1)\big\}\Big]\\
&=-6Q_1+4\alpha\CJ_\nn(Q_1)
-4\Big(B_0\cdot Q_1+B_1\cdot{Q_0}+\frac13B_1-B_1:{M^{(4)}_{Q_0}}-\frac32Q_1\Big)\\
&=-4\Big(\frac{1}{3}(B_1-\alpha Q_1)+(B_1-\alpha Q_1)\cdot Q_0-(B_1-\alpha Q_1):{M^{(4)}_{Q_0}}\Big)\\
&=2\CL_\nn(B_1)
\end{align*}
where in the last equality we used the fact that $\big(B_1-\alpha Q_1\big)\cdot Q_0=Q_0\cdot\big(B_1-\alpha Q_1\big)$,
which follows from (\ref{eq:conseq-symm}) and $Q_1=\CQ_\nn(B_1)$. Thus, the system (\ref{tensor-q-expan1})-(\ref{tensor-v-expan1}) can be written as
\begin{align}\label{tensor-q-expan2}
\frac{\pa{{Q_0}}}{\pa{t}}+\vv_0\cdot\nabla{Q_0}=&
2\CL_\nn(B_1)+2G\alpha\CJ_\nn(\Delta Q_0)+\big(\CM_{Q_0}(\kappa^T)+\CM_{Q_0}^T(\kappa^T)\big),\\
\frac{\pa{\vv_0}}{\pa{t}}+\vv_0\cdot\nabla\vv_0=&-\na p_0+
\frac{\gamma}{Re}\Delta\vv_0+\frac{1-\gamma}{2Re}\nabla\cdot(\DD_0:M_{Q_0}^{(4)})\nonumber\\
&+\frac{1-\gamma}{Re}\bigg\{-\nabla\cdot\Big[\CL_\nn(B_1)
+{G\alpha}\CM_{Q_0}(\Delta{Q_0})\Big]
+\alpha \frac{G}{2}{Q_0}:\nabla\Delta{Q_0}\bigg\}.\label{tensor-v-expan2}
\end{align}

\subsection{The derivation of Ericksen-Leslie equation}

We show that the system (\ref{tensor-q-expan2})-(\ref{tensor-v-expan2}) is equivalent to the Ericksen-Leslie equation (we replace $\vv_0$
by $\vv$ for convenience)
\begin{eqnarray}\label{eq:EL-new}
\left\{
\begin{split}
&\vv_t+\vv\cdot\nabla\vv=-\nabla{p}+\frac{\gamma}{Re}\Delta\vv
+\frac{1-\gamma}{Re}\nabla\cdot(\sigma^L+\sigma^E),\\
&\nn\times\big(\hh-\gamma_1\NN-\gamma_2\DD\cdot\nn\big)=0,
\end{split}\right.
\end{eqnarray}
where the Leslie stress $\sigma^L$ is given by
\begin{eqnarray*}
\sigma^L=\alpha_1(\nn\nn:\DD)\nn\nn+\alpha_2\nn\NN+\alpha_3\NN\nn+\alpha_4\DD
+\alpha_5\nn\nn\cdot\DD+\alpha_6\DD\cdot\nn\nn, \end{eqnarray*}
with the Leslie coefficients taking the values
\begin{align}
&\alpha_1=-\frac{S_4}{2},\qquad
\alpha_2=-\frac{S_2}{2}(1+\frac{1}{\lambda}),\qquad
\alpha_3=-\frac{S_2}{2}(1-\frac{1}{\lambda}),\nonumber\\
&\alpha_4=\frac{4}{15}-\frac{5}{21}S_2-\frac{1}{35}S_4,\quad
\alpha_5=\frac{1}{7}S_4+\frac{6}{7}S_2,\quad
\alpha_6=\frac{1}{7}S_4-\frac{1}{7}S_2,
\end{align}
and
\begin{align}
\gamma_1=\frac{1}{\frac1{3S_2}+\frac2{3S_2^2}-\frac2{S_2^2\alpha}},\quad\gamma_2=-S_2,
\quad\lambda\triangleq-\frac{\gamma_2}{\gamma_1}=\frac13+\frac2{3S_2}-\frac{2}{S_2\alpha}.
\end{align}
The Ericksen stress $\sigma^E$ is given by
\begin{eqnarray*}
\sigma^E=-\frac{\partial{E_F}}{\partial(\nabla\nn)}\cdot(\nabla\nn)^T,
\end{eqnarray*}
with $E_F=\frac12\alpha GS_2^2|\nabla\nn|^2.$

First of all, for any $\nn\nn^{\bot}+\nn^{\bot}\nn\in \mathrm{Ker}~\mathcal{L}_{\nn}$ where $\nn\cdot\nn^\bot=0$, we have
\begin{align}\label{eq:4.7}
&\Big\langle\frac{\pa{{Q_0}}}{\pa{t}}+\vv_0\cdot\nabla{Q_0}-
2G\alpha\CJ_\nn(\Delta Q_0)-\big(\CM_{Q_0}(\kappa^T)+
\CM_{Q_0}^T(\kappa^T)\big), \nn\nn^{\bot}+\nn^{\bot}\nn\Big\rangle=0.
\end{align}
By direct computations, we obtain
\begin{align*}
&\Big\langle\frac{\pa{{Q_0}}}{\pa{t}}+\vv_0\cdot\nabla{Q_0}, \nn\nn^{\bot}+\nn^{\bot}\nn\Big\rangle
=2S_2\big(\frac{\partial\nn}{\partial{t}}+\vv_0\cdot\nabla\nn\big)\cdot\nn^{\bot},\\
&\big\langle\Delta{Q_0}, \nn\nn^{\bot}+\nn^{\bot}\nn\big\rangle=2S_2\nn^\bot\cdot\Delta\nn,\\
&\big\langle(\Delta{Q_0}\cdot{Q_0}+Q_0\cdot\Delta{Q_0}), \nn\nn^{\bot}+\nn^{\bot}\nn\big\rangle
=\frac{2}{3}S_2^2\Delta\nn\cdot\nn^{\bot},\\
&\big\langle\Delta{Q_0}:M_{Q_0}^{(4)}, \nn\nn^{\bot}+\nn^{\bot}\nn\big\rangle
=\frac{1}{\alpha}\big\langle\Delta{Q_0},\quad(\nn\nn^{\bot}+\nn^{\bot}\nn)\big\rangle
=\frac{2S_2}{\alpha}\nn^\bot\cdot\Delta\nn,\\
&\Big\langle\big(\kappa_0\cdot{Q_0}+Q_0\cdot\kappa_0^T+\frac{2}{3}\DD_0
-2\kappa_0:M_{Q_0}^{(4)}\big), \nn\nn^{\bot}+\nn^{\bot}\nn\Big\rangle\nonumber\\
&=\Big\langle\big((\DD_0-\BOm_0)\cdot{Q_0}+Q_0\cdot(\DD_0+\BOm_0)+\frac{2}{3}\DD_0
-2\DD_0:M_{Q_0}^{(4)}\big), \nn\nn^{\bot}+\nn^{\bot}\nn\Big\rangle\nonumber\\
&=-2S_2\nn^\bot\cdot(\BOm_0\cdot\nn)+\frac23S_2\nn^\bot\cdot(\DD_0\cdot\nn)
+\frac43\nn^\bot\cdot(\DD_0\cdot\nn)-\frac4\alpha\nn^\bot\cdot(\DD_0\cdot\nn),
\end{align*}
which along with (\ref{eq:4.7}) gives
\begin{align*}
2S_2\big(\frac{\partial\nn}{\partial{t}}+\vv_0\cdot\nabla\nn\big)\cdot\nn^{\bot}
-\frac{4}{3}\alpha GS_2\nn^\bot\cdot\Delta\nn
-\frac{2}{3}\alpha GS_2^2\Delta\nn\cdot\nn^{\bot}
+4GS_2\nn^\bot\cdot\Delta\nn\\
+2S_2\nn^\bot\cdot(\BOm_0\cdot\nn)-\frac23S_2\nn^\bot\cdot(\DD_0\cdot\nn)
-\frac43\nn^\bot\cdot(\DD_0\cdot\nn)+\frac4\alpha\nn^\bot\cdot(\DD_0\cdot\nn)=0,
\end{align*}
or
\begin{align*}
\Big(\frac{\partial\nn}{\partial{t}}+\vv_0\cdot\nabla\nn+\BOm_0\cdot\nn\Big)\cdot\nn^{\bot}
+\alpha GS_2\Big(-\frac{2}{3S_2}
-\frac{1}{3}
+\frac{2}{\alpha S_2}\Big)\nn^\bot\cdot\Delta\nn\\
-\Big(\frac13+\frac2{3S_2}-\frac2{S_2\alpha}\Big)\nn^\bot\cdot(\DD_0\cdot\nn)=0.
\end{align*}
Thus we derive the equation of $\nn$:
\begin{align}
\nn\times\Big(\hh-\gamma_1\Big(\frac{\partial\nn}{\partial{t}}+\vv_0\cdot\nabla\nn+\BOm_0\cdot\nn\Big)-\gamma_2\DD_0\cdot\nn\Big)=0.
\end{align}

In order to derive the equation of $\vv_0$, let us compute that
\begin{align*}
&\frac{1}{2}\DD_0:M_{Q_0}^{(4)}-\CL_\nn(B_1)-{G\alpha}\CM_{Q_0}(\Delta{Q_0})\\
&=\frac{1}{2}\DD_0:M_{Q_0}^{(4)}-\CL_\nn(B_1)-\alpha\Big(\frac{1}{3}G\Delta{Q_0}+G{Q_0}\cdot\Delta{Q_0}
-G\Delta{Q_0}:M_{Q_0}^{(4)}\Big)\\
&=\frac{1}{2}\DD_0:M_{Q_0}^{(4)}-\CL_\nn(B_1)-\alpha\Big(\frac{1}{3}G\Delta{Q_0}+\frac G2\big({Q_0}\cdot\Delta{Q_0}
+\Delta{Q_0}\cdot{Q_0}\big)-G\Delta{Q_0}:M_{Q_0}^{(4)}\Big)\\
&\qquad-\alpha \frac{G}{2}\big({Q_0}\cdot\Delta{Q_0}
-\Delta{Q_0}\cdot{Q_0}\big)\\
&=\frac{1}{2}\DD_0:M_{Q_0}^{(4)}-\frac12\Big(\frac{\pa{{Q_0}}}{\pa{t}}+\vv_0\cdot\nabla{Q_0}
-\big(\kappa_0\cdot{Q_0}+Q_0\cdot\kappa_0^T+\frac{2}{3}\DD_0
-2\kappa_0:M_{Q_0}^{(4)}\big)\Big)\\
&\qquad-\alpha\frac G2\big({Q_0}\cdot\Delta{Q_0}
-\Delta{Q_0}\cdot{Q_0}\big)\\
&=-\frac12\Big(S_2(\NN\nn+\nn\NN)-S_2(\DD_0\cdot\nn\nn+\nn\DD_0\cdot\nn)+\frac{2(S_2-1)}{3}\DD_0\\
&\quad+\big(S_4(\nn\nn:\DD)\nn\nn+\frac{S_2-S_4}{7}\big((\nn\nn:\DD)\II
+2(\nn\DD\cdot\nn+\DD\cdot\nn\nn)\big)\nonumber\\
&\quad+\big(\frac{2S_4}{35}-\frac{4S_2}{21}
+\frac{2}{15}\big)\DD\big)\Big)-\frac G 2\alpha S_2^2(\nn\Delta\nn-\Delta\nn\nn)\\
&=-\frac{S_4}{2}(\DD:\nn\nn)\nn\nn-\frac{S_2-S_4}{14}
(\DD:\nn\nn)\II-\frac{S_2}{2}(1+\frac{1}{\lambda})\nn\NN
-\frac{S_2}{2}(1-\frac{1}{\lambda})\NN\nn\nonumber\\
&\quad\quad+\big(\frac{4}{15}-\frac{5}{21}S_2-\frac{1}{35}S_4\big)\DD
+\big(\frac{1}{7}S_4+\frac{6}{7}S_2\big)\nn\nn\cdot\DD
+\big(\frac{1}{7}S_4-\frac{1}{7}S_2\big)\DD\cdot\nn\nn\\
&\equiv -\frac{S_2-S_4}{14}(\DD:\nn\nn)\II+\sigma^L,
\end{align*}
where $\sigma^L$ is the Leslie stress, and
\begin{align*}
\frac12\alpha G{Q_0}:\nabla\Delta{Q_0}=&\frac12\nabla(\alpha GQ_0:\Delta{Q_0})-\frac12\alpha GS_2^2\partial_l(n_in_j)\partial_k^2(n_in_j)\\
=&\frac12\nabla(\alpha GQ_0:\Delta{Q_0})-\alpha GS_2^2\partial_ln_i\partial_k^2n_i\\
=&\frac12\nabla(\alpha GQ_0:\Delta{Q_0})+GS_2^2\partial_l\partial_kn_i\partial_kn_i
-\alpha GS_2^2\partial_k\Big(\partial_kn_i\partial_ln_i\Big)\\
=&\frac12\nabla\Big(\alpha GQ_0:\Delta{Q_0}+\alpha GS_2^2(\nabla\nn)^2\Big)
-\alpha GS_2^2\nabla\cdot(\nabla\nn\odot\nabla\nn)\\
\equiv&-\na p_1+\na\cdot \sigma^E,
\end{align*}
where $\sigma^E=-\alpha GS_2^2(\nabla\nn\odot\nabla\nn)$ is the Ericksen stress. We denote
\beno
\sigma=\sigma^L+\sigma^E,\quad p=p_0+p_1+\frac{S_2-S_4}{14}(\DD:\nn\nn).
\eeno
Thus we derive the equation of $\vv_0$:
\ben
\frac{\pa{\vv_0}}{\pa{t}}+\vv_0\cdot\nabla\vv_0=-\na p+\frac{\gamma}{Re}\Delta\vv_0+\na\cdot\sigma.
\een
This completes the derivation of the Ericksen-Leslie equation. Note that the constraints
(\ref{Leslie relation}) and (\ref{Leslie-coeff}) are naturally satisfied.
\begin{remark}
Comparing with the Ericksen-Leslie system directly derived from the Doi-Onsager theory,
the only difference is that the value of $\gamma_1$ (and hence $\lambda, \alpha_2, \alpha_3$)
has been changed. This change comes from the Bingham closure, which approximates a general
probability distribution function by a specific distribution function.
\end{remark}

\section{The hard-core potential case}
In \cite{HLWZ}, the general Oseen-Frank energy is derived from the Onsager theory with
hard-core potential. Here we present a sketch of derivation of the Ericksen-Leslie equation
with general Oseen-Frank energy
\begin{align}
E_{OF}(\textbf{n}, \nabla \textbf{n}) = &\frac{1}{2}
K_1 (\mathrm{div} \textbf{n})^2 + \frac{1}{2}K_2 (\textbf{n}\cdot \nabla
\times \textbf{n})^2 + \frac{1}{2}K_3 (\textbf{n} \times \nabla \textbf{n})^2.\nonumber
\end{align}

We introduce the fourth order traceless tensor $Q_4=Q_4(Q)$ defined by
\begin{align}
Q_{4\alpha\beta\gamma\delta}=&\int_\BS\Big({m}_{\alpha}{m}_{\beta}{m}_{\gamma}{m}_{\mu}
-\frac{1}{7}\big[m_{\alpha}m_{\beta}\delta_{\gamma\mu}+
m_{\gamma}m_{\mu}\delta_{\alpha\beta}+m_{\alpha}m_{\gamma}\delta_{\beta\mu}
+m_{\beta}m_{\mu}\delta_{\alpha\gamma}\nonumber\\
&\qquad+m_{\alpha}m_{\mu}\delta_{\beta\gamma}+
m_{\beta}m_{\gamma}\delta_{\alpha\mu}\big]+\frac{1}{35}\big[\delta_{\alpha\beta}\delta_{\gamma\mu}
+\delta_{\alpha\gamma}\delta_{\beta\mu}+\delta_{\alpha\mu}\delta_{\beta\gamma}\big]\Big)f_Q\ud\mm,\nonumber
\end{align}
and the energy $E$ defined by
\begin{align}
E(Q,\nabla Q)=E_b(Q)+ E_e(Q,\nabla Q),\nonumber
\end{align}
where
\begin{align*}
&E_b(Q)=\int_{\BR}\Big(-\ln{Z}_Q+Q:B_Q-
\frac{\alpha}{2}|Q|^2\Big)\ud\xx,\\
&E_e(Q,\nabla Q)=\frac\ve2\int_\BR\bigg\{
J_1|\nabla Q|^2+J_2|\nabla Q_4|^2+J_3\Big(\partial_i(Q_{ik})\partial_j(Q_{jk})+\partial_i(Q_{jk})\partial_j(Q_{ik})\Big)\nonumber\\
&\qquad+J_4\Big(\partial_i(Q_{4iklm})\partial_j
(Q_{4jklm})+\partial_i(Q_{4jklm})\partial_j(Q_{4iklm})\Big)+J_5\partial_i(Q_{4ijkl})
\partial_j(Q_{kl})\bigg\}\ud\xx.\nonumber
\end{align*}
The coefficients $J_i(i=1,\cdots,5)$ are explicitly calculated in terms of the molecular parameters in \cite{HLWZ}.
Here we set them to be general.
Let
\begin{align}
\mu_Q=\frac{\delta E(Q,\nabla Q)}{\delta Q}=B_Q-\alpha Q+\frac{\delta E_e}{\delta Q}.\nonumber
\end{align}
In such case, the $Q$-tensor equation becomes
\begin{align}
\frac{\pa{{Q}}}{\pa{t}}+\vv\cdot\nabla{Q}&=\frac{\ve}{De}\mathcal{N}(\mu_Q)
+\CM_Q(\kappa^T)+\CM_Q^T(\kappa^T)+\frac{2}{De}\Big(\CM_Q(\mu_Q)
+\CM_Q^T(\mu_Q)\Big),\label{eq:Q-general}\\\nonumber
\frac{\pa{\vv}}{\pa{t}}+\vv\cdot\nabla\vv&=-\nabla{p}
+\frac{\gamma}{Re}\Delta\vv+\frac{1-\gamma}{2Re}\nabla\cdot(\DD:M_Q^{(4)})\\
&\qquad\quad+\frac{1-\gamma}{DeRe}\Big(2\nabla\cdot\CM_Q(\mu_Q)+\mu_Q:\nabla{Q}\Big).\label{eq:v-general}
\end{align}
Then we can derive the Ericksen-Leslie equation from (\ref{eq:Q-general}) and (\ref{eq:v-general}) by taking small
Debourah number limit, where the Leslie coefficients and $\gamma_1$, $\gamma_2$  keep the same, but the Ericksen energy $E_F$ is replaced by the general
Oseen-Frank energy $E_{OF}$  with the elastic coefficients $K_1, K_2, K_3$ given by
\begin{align*}
&K_1=2S_2^2(J_1+J_3)+S_4^2(\frac{16}{7}J_2+\frac{92}{49}J_4)-\frac67J_5S_2S_4,\\
&K_2=2S_2^2J_1+S_4^2(\frac{16}{7}J_2+\frac{12}{49}J_4)-\frac27J_5S_2S_4,\\
&K_3=2S_2^2(J_1+J_3)+S_4^2(\frac{16}{7}J_2+\frac{120}{49}J_4)+\frac87J_5S_2S_4,
\end{align*}
where $S_2,~S_4$ are defined by (\ref{order-parameter}). We omit the detailed derivation here.

Let us conclude this section by some comparisons with two dynamical Q-tensor models mentioned in the introduction (see also Remarks in section 2.3).
Our dynamical $Q$-tensor theory could written in the following form similar to
(\ref{eq:Q-general-intro}):
\begin{align*}
\frac{\pa{{Q}}}{\pa{t}}+\vv\cdot\nabla{Q}&=D^{trans}(\mu_Q)+D^{rot}(\mu_Q)+F(Q,\DD)
+\BOm\cdot\mu_Q-\mu_Q\cdot\BOm,\\
\frac{\pa{\vv}}{\pa{t}}+\vv\cdot\nabla\vv&=-\nabla{p}+\nabla\cdot\big(\sigma^{dis}+\sigma^{s}+\sigma^a+\sigma^{d}\big),\\
\nabla\cdot\vv&=0.
\end{align*}
Here the additional term
$$D^{trans}=\frac{\ve}{De}\mathcal{N}(\mu_Q)$$ accounts for the translational diffusion, which is not considered in Beris-Edwards's and Qian-Sheng's models.
The terms $\sigma^a$ and $\sigma^d$, module some constants, are all the same as those in Beris-Edwards's and Qian-Sheng's models:
\begin{align}
\sigma_{ij}^d=\frac{\partial E}{\partial (Q_{kl,j})}Q_{kl,i},\quad
\sigma^a=Q\cdot\mu_Q-\mu_Q\cdot Q.\nonumber
\end{align}
Here it should be noticed that $\mu_Q:\nabla{Q}$ is actually the same as
 $\partial_j\big(\frac{\partial E}{\partial (Q_{kl,j})}Q_{kl,i}\big)$ module a pressure term, and
$$(\CM-\CM^T)(\mu_Q)=Q\cdot\mu_Q-\mu_Q\cdot Q.$$
Our rotational diffusion term is derived from Doi's kinetic theory, which
takes the form:
$$D^{rot}=\frac{2}{De}(\CM_Q+\CM_Q^T)(\mu_Q).$$
The two conjugated terms $F(Q,\DD)$ and $\sigma^s=-\frac{1-\gamma}{Re De}F(Q,\mu_Q)$ are given by
$$F(Q,A)=(\CM_Q+\CM_Q^T)(A).$$
The additional dissipation stress is given by
$$\sigma^{dis}=\frac{2\gamma}{Re}\DD+\frac{1-\gamma}{2Re}\DD:M_Q^{(4)}.$$


%

\bigskip

\noindent {\bf Acknowledgments.}
P. Zhang is partly supported by NSF of China under Grant 50930003 and 21274005.
Z. Zhang is partially supported by NSF of China under Grant 10990013 and 11071007,
Program for New Century Excellent Talents in University and Fok Ying Tung Education Foundation.

\end{document}